\documentclass[11pt, a4paper]{article}
\usepackage{times}
\usepackage{a4wide}
\usepackage{hyperref}
\usepackage[british]{babel}
\usepackage{enumerate}
\usepackage{amsmath, amscd, amsfonts, amsthm, amssymb, latexsym, stmaryrd}
\usepackage[T1]{fontenc}
\usepackage[latin1]{inputenc}

\newtheorem{thm}{Theorem}[section]
\newtheorem{lem}[thm]{Lemma}

\newtheorem{prop}[thm]{Proposition}
\newtheorem{cor}[thm]{Corollary}
\newtheorem{notat}[thm]{Notation}
\newtheorem{ass}[thm]{Assumption}

\newcommand{\End}{{\rm End}}
\newcommand{\Hom}{{\rm Hom}}
\newcommand{\GL}{\mathrm{GL}}
\newcommand{\SL}{\mathrm{SL}}

\newcommand{\Image}{\mathrm{im}}
\newcommand{\Katz}{\mathrm{Katz}}
\newcommand{\cl}{\mathrm{cl}}
\newcommand{\id}{\mathrm{id}}
\newcommand{\pr}{\mathrm{pr}}

\DeclareMathOperator{\Gal}{Gal}
\DeclareMathOperator{\Frob}{Frob}

\DeclareMathOperator{\Mat}{Mat}
\DeclareMathOperator{\charpoly}{charpoly}

\newcommand{\cG}{\mathcal{G}}
\newcommand{\cI}{\mathcal{I}}
\newcommand{\fm}{\mathfrak{m}}
\newcommand{\CC}{\mathbb{C}}
\newcommand{\FF}{\mathbb{F}}
\newcommand{\NN}{\mathbb{N}}
\newcommand{\QQ}{\mathbb{Q}}
\newcommand{\TT}{\mathbb{T}}
\newcommand{\ZZ}{\mathbb{Z}}
\newcommand{\Qbar}{\overline{\QQ}}
\newcommand{\Fbar}{\overline{\FF}}
\newcommand{\rhobar}{\overline{\rho}}

\newcommand{\diam}[1]{{\langle #1 \rangle}}

\newcommand{\mat}[4]{
 \left(  \begin{smallmatrix} #1 & #2 \\ #3 & #4 \end{smallmatrix} \right)}

\begin{document}

\title{On Galois Representations of Weight One}
\author{Gabor Wiese}
\maketitle

\begin{abstract}
A two-dimensional Galois representation into the Hecke algebra of Katz modular
forms of weight one over a finite field of characteristic~$p$ is constructed and is
shown to be unramified at~$p$ in most cases.
\bigskip

\end{abstract}

\maketitle

\section{Introduction}

Let $g=\sum_{n=1}^\infty a_n q^n$
be a holomorphic cuspidal Hecke eigenform of weight~$1$ on $\Gamma_1(N)$ with Dirichlet
character~$\epsilon$. Deligne and Serre \cite{deligne-serre} constructed a Galois representation
$$ \rho_g: \Gal(\Qbar/\QQ) \to \GL_2(\CC),$$
which is irreducible, unramified outside~$N$ and characterised by
$\rho_g(\Frob_\ell)$ having characteristic polynomial $X^2 - a_\ell X + \epsilon(\ell)$
for all primes~$\ell \nmid N$. Here, and throughout, $\Frob_\ell$ denotes an arithmetic
Frobenius element.
Let now $p \nmid N$ be a prime number. Reducing $\rho_g$ modulo (a prime above)~$p$
and semisimplifying yields a Galois representation
$$ \rhobar_g: \Gal(\Qbar/\QQ) \to \GL_2(\Fbar_p),$$
which is still unramified outside $N$ (in particular, at~$p$)
and still satisfies the respective formula for the characteristic polynomials
at all unramified primes. In fact, $\rhobar_g$ only depends on the reduction
of (the coefficients of)~$g$ modulo (a prime above)~$p$.

In this article, we shall work more generally and
study normalised cuspidal Katz eigenforms~$g$ over~$\Fbar_p$
of weight~$1$ on $\Gamma_1(N)$ with Dirichlet character~$\epsilon$
and $q$-expansion (at~$\infty$) $\sum_{n=1}^\infty a_n q^n$ for $p \nmid N$
(see Section~\ref{sec:mf}).
Unlike when the weight is at least~$2$, not all such $g$ can be obtained by reducing
holomorphic weight~$1$ forms. The first such nonliftable example was found by Mestre
(see Appendix~A of~\cite{edix-jussieu}).
Nevertheless, $g$ also has an attached Galois representation $\rhobar_g$ which is unramified
outside $Np$ and such that the characteristic polynomials at unramified primes look as before.
In their study of companion forms, Gross, Coleman and Voloch proved that $\rhobar_g$ is
also unramified at~$p$ in almost all cases.

\begin{prop}[Gross, Coleman, Voloch]\label{prop:weightone}
If $p = 2$, assume that $a_p^2\neq 4 \epsilon(p)$ (i.e.\ $a_2 \neq 0$).
Then the residual representation $\rhobar_g$ is unramified at~$p$.
\end{prop}

\begin{proof}
The case $p>2$ is treated by \cite{coleman-voloch} without any assumption on $a_p$.
\cite{gross} proves the result for all~$p$ (i.e.\ including $p=2$), but with the extra
condition, subject to some unchecked compatibilities, which have now been settled
by Bryden Cais in~\cite{cais-thesis}, Chapter~10.
\end{proof}

In this article we give a somewhat different proof and remove the condition in the case $p=2$.
The main objective, however, is to extend the representation $\rhobar_g$ to a
$2$-dimensional Galois representation
with coefficients in the weight~$1$ Hecke algebra and to show, in most cases,
that it is also unramified outside $N$, in particular at~$p$, with the natural characteristic
polynomials at all unramified primes.

We now introduce the notation necessary to state the main result.
Let $S_1(\Gamma_1(N);\Fbar_p)_\Katz$ be the $\Fbar_p$-vector space of
cuspidal Katz modular forms of weight~$1$ on~$\Gamma_1(N)$ over~$\Fbar_p$ and
let $\TT_1$ be the Hecke algebra acting on it,
i.e.\ the $\FF_p$-subalgebra inside $\End_{\Fbar_p}(S_1(\Gamma_1(N);\Fbar_p)_\Katz)$
generated by all Hecke operators~$T_n$ for $n \in \NN$ (see also Section~\ref{sec:mf}).
Let $\fm$ be the maximal ideal of~$\TT_1$ defined as the kernel of the ring homomorphism
$\TT_1 \xrightarrow{T_n \mapsto a_n} \Fbar_p$.
Let $\TT_{1,\fm}$ denote the localisation of $\TT_1$ at~$\fm$.
For the representation $\rhobar_g$ we shall also write $\rhobar_\fm$.

\begin{thm}\label{thm:main}
Assume that $\rhobar_\fm$ is irreducible
and that, if $\rhobar_\fm$ is unramified at~$p$, then $\rhobar_\fm(\Frob_p)$ is not scalar.

Then there is a Galois representation
$$ \rho_\fm: \Gal(\Qbar/\QQ) \to \GL_2(\TT_{1,\fm})$$
which is unramified outside~$N$ and such that for all primes $\ell \nmid N$ (including $\ell=p$)
the characteristic polynomial of $\rho_\fm(\Frob_\ell)$ is
$X^2 - T_\ell X + \diam \ell \in \TT_{1,\fm}[X]$.
\end{thm}

Note that we are not assuming that $\rhobar_\fm$ is unramified at~$p$, but, that this can
be deduced from the theorem.
This removes the condition in the case $p=2$ from Proposition~\ref{prop:weightone}.

\begin{cor}\label{cor:main}
The representation $\rhobar_g$ is unramified outside~$N$
and the characteristic polynomial of $\rhobar_g(\Frob_\ell)$
equals $X^2 - a_\ell X + \epsilon(\ell)$
for all primes $\ell \nmid N$, including $\ell = p$.
\end{cor}

\begin{proof}
If $\rhobar_\fm = \rhobar_g$ is reducible, then it is automatically unramified at~$p$ as it is semisimple.
If $\rhobar_\fm$ is irreducible, the result follows by reducing $\rho_\fm$
(from Theorem~\ref{thm:main}) modulo~$\fm$.
\end{proof}

The proof of Theorem~\ref{thm:main} will be given in Section~\ref{sec:proof}
and we will illustrate the theorem with examples
in Section~\ref{sec:examples}.
The essential point that makes the proof work is that cusp forms of weight~$1$
over~$\Fbar_p$ sit in weight~$p$ in two different ways; on $q$-expansions the situation is
precisely the same as in the theory of oldforms, when passing from level~$N$ to level~$Np$.
Let us call this `doubling'. We shall see that it leads to a `doubling of Hecke
algebras' and finally to a `doubling of Galois representations'. It is from the latter that
we deduce the main statement.

In the proof of Theorem~\ref{thm:main} we use essentially that $\rhobar_\fm$
satisfies multiplicity one (see Section~\ref{sec:galois}); hence, the case when
$\rhobar_\fm$ is unramified at~$p$ with scalar $\rhobar_\fm(\Frob_p)$, where
multiplicity one fails by Corollary~4.5 of~\cite{multiplicities}, has to remain open here.

Since the present article was finished and first put on arXiv (arXiv:1102.2302), I made some
unsuccessful efforts to remove the multiplicity one-assumption. Since then,
Frank Calegari and David Geraghty released an impressive preprint~\cite{CG}, in which they manage
to extend Theorem~\ref{thm:main} to all cases (for odd primes~$p$)
and, moreover, achieve an $R=\TT$-theorem,
using a detailed analysis of the local deformation rings.
They also prove that the relevant multiplicity is $2$ if it is not~$1$,
completing the main result of~\cite{multiplicities}.

We finish this introduction by expressing our optimism that the methods of the present paper
might generalise to some extent to Hilbert modular forms. We intend to investigate
this in future work.

\subsection*{Acknowledgements}

I would like to thank Mladen Dimitrov for clarifying discussions and Jean-Pierre Serre
for comments and hints on analysing the examples. I also thank Kevin Buzzard for some comments
on the first version of this article.
Thanks are also due to the referee for a careful reading, pointing out a notational
ambiguity and some helpful remarks.

I acknowledge partial support during the writing of this work by the SFB/TRR 45 and the SPP 1489
of the Deutsche Forschungsgemeinschaft.

\section{Modular forms and Hecke algebras of weight one}\label{sec:mf}

In this section we provide the statements on modular forms and Hecke algebras
that are needed for the sequel. In particular, we deduce a `doubling of Hecke algebras'
from a `doubling of modular forms'.

We shall use the following notation and assumptions throughout the article.

\begin{notat}\label{notat:mf}
\begin{itemize}
\item Let $p$ be a prime number and $N\ge 5$ an integer not divisible by~$p$.

\item $\Frob_\ell$ always denotes an arithmetic Frobenius element at~$\ell$.

\item $\zeta_n$ always denotes a primitive $n$-th root of unity (for $n \in \NN$).

\item If $R$ is a ring, $M$, $N$ are $R$-modules and $S \subseteq M$ is a subset, then we put
$$\Hom_R(M,N)^{S=0} := \{ f \in \Hom_R(M,N) \;|\; f(s) = 0 \; \forall \, s \in S\}.$$
\end{itemize}
\end{notat}

\subsection*{Katz modular forms}

For the treatment in this article, it is essential to dispose of the geometric
definition of modular forms given by Katz. Since the tools we need are nicely
exposed in \cite{edix-jussieu}, we follow this article, and, in particular, we work
with Katz modular cusp forms for the moduli problem $[\Gamma_1(N)]_{\FF_p}'$
of elliptic curves with an embedding of the group scheme $\mu_N$.
We use the notation $S_k(\Gamma_1(N); \FF_p)_\Katz$ for these. Replacing $\FF_p$
by a field extension $\FF$ of~$\FF_p$, one also defines $S_k(\Gamma_1(N); \FF)_\Katz$.
By flatness,
$S_k(\Gamma_1(N); \FF)_\Katz = \FF \otimes_{\FF_p} S_k(\Gamma_1(N); \FF_p)_\Katz$.

Let
$\TT_k(\Gamma_1(N);\FF_p)_\Katz$ be the $\FF_p$-subalgebra of
$\End_{\Fbar_p}(S_k(\Gamma_1(N);\Fbar_p)_\Katz)$
generated by all Hecke operators $T_n$ and let $\TT_k'(\Gamma_1(N);\FF_p)_\Katz$
be its subalgebra generated only by those $T_n$ with $p \nmid n$.
Note that both contain the diamond operators due to the formula
$\ell^{k-1}\diam{\ell} = T_{\ell}^2 - T_{\ell^2}$ for a prime~$\ell$.
The $q$-expansion principle (see e.g.\ \cite{DI}, Theorem~12.3.4) and the formula $a_1(T_nf)=a_n(f)$
show that the pairing
$\TT_k(\Gamma_1(N);\FF_p)_\Katz \times S_k(\Gamma_1(N); \FF_p)_\Katz \to \FF_p$,
given by $\langle T,f\rangle = a_1(Tf)$ is nondegenerate and, thus, provides the
identification
\begin{equation}\label{eq:katz}
\Hom_{\FF_p}(\TT_k(\Gamma_1(N);\FF_p)_\Katz,\FF)
\xrightarrow{\varphi \mapsto \sum_{n=1}^\infty \varphi(T_n) q^n}
S_k(\Gamma_1(N); \FF)_\Katz
\end{equation}
for $\FF/\FF_p$.

\subsection*{Classical modular forms}

It is useful to point out the relation with classical holomorphic cusp forms,
for which we use the notation $S_k(\Gamma_1(N))$ and $S_k(\Gamma_1(N))_\cl$.
The corresponding Hecke algebra $\TT_k(\Gamma_1(N))_\cl$ is defined as the
$\ZZ$-subalgebra of $\End_\CC(S_k(\Gamma_1(N))_\cl)$ generated by all
Hecke operators~$T_n$.
By the existence of an integral structure and the $q$-expansion principle,
the map $\Hom_\ZZ(\TT_k(\Gamma_1(N))_\cl,\CC) \to S_k(\Gamma_1(N))_\cl$
which sends a map $\varphi$ to the Fourier series $\sum_{n=1}^\infty \varphi(T_n) q^n$
with $q=q(z)=e^{2 \pi i z}$ is an isomorphism. We let
$S_k(\Gamma_1(N);R)_\cl = \Hom_\ZZ(\TT_k(\Gamma_1(N))_\cl,R)$ for any $\ZZ$-algebra~$R$.
Note $S_k(\Gamma_1(N);\CC)_\cl = S_k(\Gamma_1(N))_\cl$.
Note also that due to the freeness and the finite generation of $\TT_k(\Gamma_1(N))_\cl$ as a
$\ZZ$-module
\begin{equation}\label{eq:classical}
R_2 \otimes_{R_1} S_k(\Gamma_1(N);R_1)_\cl \cong S_k(\Gamma_1(N);R_2)_\cl
\end{equation}
for any $R_1$-algebra $R_2$.
For a subring $R \subseteq \CC$ the $R$-module $S_k(\Gamma_1(N);R)_\cl$ agrees
with the subset of $S_k(\Gamma_1(N))$ consisting of those forms with
$q$-expansion having coefficients in~$R$, as e.g.\ in~\cite{DI}, Section~12.3.

The following proposition states that for weights at least~$2$, Katz cusp forms over~$\Fbar_p$
coincide with reductions of classical ones of the same level~$\Gamma_1(N)$.

\begin{prop}\label{prop:comparison}
Let $k \ge 2$. Assume $N \neq 1$ or $p \ge 5$.
\begin{enumerate}[(a)]
\item\label{prop:comparison:a}
There is an isomorphism
$S_k(\Gamma_1(N);\FF)_\cl \cong S_k(\Gamma_1(N);\FF)_\Katz$ which is compatible with
the Hecke operators and $q$-expansions for any $\FF/\FF_p$.
\item\label{prop:comparison:b}
The map
$\FF_p \otimes_\ZZ\TT_k(\Gamma_1(N))_\cl \xrightarrow{1 \otimes T_n \mapsto T_n}
       \TT_k(\Gamma_1(N);\FF_p)_\Katz$ is an isomorphism of rings.
\end{enumerate}
\end{prop}

\begin{proof}
(\ref{prop:comparison:a}) By Theorem~12.3.2 of~\cite{DI}
(see also Lemma~1.9 of~\cite{edix-FLT} for the cases $N<5$) one has
$$ \FF_p \otimes_{\ZZ[\frac{1}{N}]} S_k(\Gamma_1(N);\ZZ[\frac{1}{N}])_\Katz
\cong S_k(\Gamma_1(N);\FF_p)_\Katz$$
compatible with the Hecke operators. By \cite{edix-jussieu}, 4.7, one also has
$$ S_k(\Gamma_1(N);\ZZ[\frac{1}{N}])_\Katz \cong S_k(\Gamma_1(N);\ZZ[\frac{1}{N}])_\cl.$$
Both identifications respect $q$-expansions.
Invoking them together with Equation~\eqref{eq:classical} gives the statement.

(\ref{prop:comparison:b}) From~(\ref{prop:comparison:a}) it follows that
$\FF_p \otimes_\ZZ\TT_k(\Gamma_1(N))_\cl \xrightarrow{1 \otimes T_n \mapsto T_n}
\TT_k(\Gamma_1(N);\FF_p)_\Katz$
is a surjection of rings. To see it is an isomorphism it suffices to invoke
Equations \eqref{eq:katz} and~\eqref{eq:classical} to give:
\begin{multline*}
\Hom_{\FF_p}(\TT_k(\Gamma_1(N);\FF_p)_\Katz,\FF_p) \cong
S_k(\Gamma_1(N);\FF_p)_\Katz \overset{\textnormal{(\ref{prop:comparison:a})}}{\cong}\\
S_k(\Gamma_1(N);\FF_p)_\cl \cong
\Hom_{\ZZ}(\TT_k(\Gamma_1(N))_\cl,\FF_p) \cong
\Hom_{\FF_p}(\FF_p \otimes_\ZZ\TT_k(\Gamma_1(N))_\cl,\FF_p),
\end{multline*}
which is the map induced from $1 \otimes T_n \mapsto T_n$ due to the compatibility
of $q$-expansions.
\end{proof}

Note that the corresponding statement for weight $k=1$ is false.
We shall explain examples in Section~\ref{sec:examples}.
That failure is actually {\em la raison d'\^{e}tre} of this article.

\subsection*{Doubling of weight one forms}

Towards the goal of this article, the construction and study of a Galois representation
into the weight~$1$ Hecke algebra, it is necessary to increase the
weight, since weight~$1$ is not a cohomological weight. The increased weight will
enable us to see the Galois representation on the Jacobian of a modular curve,
thus, permitting the use of geometric tools.

We shall map weight~$1$ forms into weight~$p$. This can be done in two different ways:
multiplying by the Hasse invariant $A$ (a modular form over $\FF_p$ of weight $p-1$
with $q$-expansion~$1$); the Frobenius $F(f)= f^p$. The former does not change the
$q$-expansion and the latter maps $\sum_{n=1}^\infty a_n q^n$ to $\sum_{n=1}^\infty a_n q^{np}$.
The formula for~$F$ is clear for modular forms over~$\FF_p$; if we work with coefficients
in $\FF/\FF_p$, then as in~\cite{edix-jussieu} we take the $\FF$-linear extension of~$F$.
Note that on the level of $q$-expansions, the two maps $A$ and $F$ correspond precisely to
the two degeneracy maps from level $N$ to~$Np$. Hence, weight one forms in weight~$p$
are very analogous to oldforms. That is the price to pay for the use of geometric tools.

Let $\FF/\FF_p$ and consider the map
\begin{equation}\label{eq:psi}
\Psi: \big(S_1(\Gamma_1(N);\FF)_\Katz\big)^{\oplus 2} \rightarrow S_p(\Gamma_1(N);\FF)_\Katz,\;\;\;
(f,g) \mapsto A(f) + F(g) = Af + g^p.
\end{equation}
By Proposition~4.4 of~\cite{edix-jussieu} this is an injection. We shall write $T_p$
for the Hecke operator in weight~$1$ and $U_p$ for the one in weight~$p$. According to
Equation~(4.2) of loc.~cit.\ one has
\begin{equation}\label{eq:hecke}
   \diam a \circ \Psi = \Psi \circ \mat{\diam a}00{\diam a}, \;\;\;
   T_n \circ \Psi = \Psi \circ \mat{T_n}00{T_n}, \;\;\;
   U_p \circ \Psi = \Psi \circ \mat{T_p}{\id}{-\diam p}{0}
\end{equation}
for $p\nmid n$ and $a \in \ZZ/N\ZZ^\times$.

\subsection*{The weight one Hecke algebra and doubling of Hecke algebras}

From now on we use the abbreviations
$\TT_k$ and $\TT_k'$ for $\TT_k(\Gamma_1(N);\FF_p)_\Katz$
and $\TT_k'(\Gamma_1(N);\FF_p)_\Katz$, respectively.
Note that Equation~\eqref{eq:hecke} implies that $\TT_p'$ acts on
$S_1(\Gamma_1(N);\FF)_\Katz$ via the embedding given by~$A$.
In particular, mapping $T_n \mapsto T_n$
for $p \nmid n$ defines a ring surjection $\TT_p' \twoheadrightarrow \TT_1'$.
Define
$$ I := \TT_p' \cap U_p \TT_p',$$
where the intersection is taken inside $\TT_p$.
We shall see in Corollary~\ref{cor:ideal}~(\ref{cor:ideal:c})
that $I$ is the kernel of the previous surjection.

\begin{lem}\label{lem:ideal}
\begin{enumerate}[(a)]
\item\label{lem:ideal:a}
Inside $S_p(\Gamma_1(N);\FF)_\Katz$ the equality
$\Hom_{\FF_p}(\TT_p,\FF)^{\TT_p'=0} = FS_1(\Gamma_1(N);\FF)_\Katz$
holds (via Equation~\eqref{eq:katz}).

\item\label{lem:ideal:b}
Inside $S_p(\Gamma_1(N);\FF)_\Katz$ the equality
$U_p\Hom_{\FF_p}(\TT_p,\FF)^{\TT_p'=0} = AS_1(\Gamma_1(N);\FF)_\Katz$
holds (via Equation~\eqref{eq:katz}).
\end{enumerate}
\end{lem}

\begin{proof}
(\ref{lem:ideal:a}) As $\TT_p'$ is generated as $\FF_p$-vector space by the Hecke
operators $T_n$ for $p \nmid n$, the left hand side is equal to
$\{f \in S_p(\Gamma_1(N);\FF)_\Katz \;|\; a_n(f)=0 \; \forall n \textnormal{ s.t.\ } p \nmid n\}$.
As this is precisely the kernel of $\Theta$ defined in~\cite{katz-modp}, part~(3)
of the main theorem of loc.~cit.\
implies that it is equal to $FS_1(\Gamma_1(N);\FF)_\Katz$.
(\ref{lem:ideal:b}) follows from Equation~\eqref{eq:hecke}, namely one has $U_pF=A$.
\end{proof}

\begin{prop}\label{prop:ideal}
\begin{enumerate}[(a)]
\item\label{prop:ideal:e}
Let $f \in S_1(\Gamma_1(N);\FF)_\Katz$ satisfy $a_n(f)=0$ whenever $p \nmid n$. Then $f=0$.
\item\label{prop:ideal:a}
$\TT_1' = \TT_1$.
\item\label{prop:ideal:b}
$\TT_p = \TT_p' + U_p \TT_p'$ as $\TT_p'$-modules.
\item\label{prop:ideal:c}
There are $T,D \in \TT_p'$ such that $U_p^2 - T U_p + D=0$ in $\TT_p$.
\item\label{prop:ideal:d}
$I$ is an ideal of $\TT_p$.
\end{enumerate}
\end{prop}

\begin{proof}
(\ref{prop:ideal:e}) The theorem of~\cite{katz-modp} already used in the previous proof
gives a contradiction for $f \neq 0$, since that $f$ is in the kernel of $\Theta$
and has weight~$1$, so that it would have to come from an even smaller weight under Frobenius,
which is impossible.

(\ref{prop:ideal:a}) If $\TT_1/\TT_1'$ were nonzero, then
$S_1(\Gamma_1(N);\FF_p)^{\TT_1'=0} = \Hom_{\FF_p}(\TT_1,\FF_p)^{\TT_1'=0}$
would be nonzero and, hence, there would be a nonzero form
$f \in S_1(\Gamma_1(N);\FF_p)_\Katz$ such that $a_n(f)=0$ whenever $p \nmid n$. This
is, however, excluded by~(\ref{prop:ideal:e}).

(\ref{prop:ideal:b})
Let $g \in S_p(\Gamma_1(N);\FF_p)^{\TT_p' + U_p \TT_p'=0}=\Hom_{\FF_p}(\TT_p, \FF_p)^{\TT_p' + U_p \TT_p'=0}$. Now $g$
satisfies $a_n(g) = 0$ whenever $p^2 \nmid n$.
Thus, there is $f \in S_1(\Gamma_1(N);\FF_p)_\Katz$ such
that $Ff = g$ (again by \cite{katz-modp}) satisfying $a_n(f)=0$ whenever
$p \nmid n$, so that by (\ref{prop:ideal:e}) it is zero, implying the claim.

(\ref{prop:ideal:c}) This is immediate from~(\ref{prop:ideal:b}).

(\ref{prop:ideal:d}) Let $x \in \TT_p' \cap U_p \TT_p'$. Thus, there is $y \in \TT_p'$ such that
$x = U_p y$. We have
$$ U_p x = U_p^2 y = TU_p y - D y = Tx - Dy \in \TT_p',$$
whence $U_px \in I$.
\end{proof}

Let $\fm'$ be a maximal ideal of $\TT_p'$. By $\TT_{p,\fm'}$ and $\TT_{p,\fm'}'$
we denote localisation at~$\fm'$. We also use similar notation in similar circumstances.

\begin{lem}\label{lem:max}
Let $\fm'$ be a maximal ideal of $\TT_p'$.
\begin{enumerate}[(a)]
\item\label{lem:max:a}
The following statements are equivalent:
\begin{enumerate}[(i)]
\item\label{lem:max:a:i}
$\TT_ {p, \fm'}' \neq \TT_{p,\fm'}$.
\item\label{lem:max:a:ii}
There is a normalised eigenform $g \in S_1(\Gamma_1(N);\Fbar_p)_\Katz$ with $q$-expansion
$\sum_{n=1}^\infty a_n q^n$ such that the map
$\TT_p' \xrightarrow{T_n \mapsto a_n} \Fbar_p$ defines a ring homomorphism with kernel~$\fm'$.
\end{enumerate}
If the equivalent statements hold, then we say that {\em $\fm'$ comes from weight~$1$}.

\item\label{lem:max:b}
We have $\TT_{p,\fm'} \cong \prod_{i=1}^n \TT_{p,\tilde{\fm}_i}$ with $n \in \{1,2\}$,
where the $\tilde{\fm}_i$ are the maximal ideals of $\TT_p$ containing~$\fm'$.

If one of the $\tilde{\fm}_i$ is ordinary (meaning that $U_p \in \TT_{p,\tilde{\fm}_i}^\times$),
then all are and we say that $\fm'$ is {\em ordinary}.

Suppose now that $\fm'$ comes from weight one with $g \in S_1(\Gamma_1(N);\Fbar_p)_\Katz$
as in~(\ref{lem:max:a}). Then $\fm'$ is ordinary.
Furthermore, $n=2$ if and only if the polynomial $X^2 - a_p(g)X + \epsilon(p)$
has two distinct roots in $\TT_p'/\fm'$.
\end{enumerate}
\end{lem}

\begin{proof}
(\ref{lem:max:a}) Statement~(\ref{lem:max:a:i}) means that $U_p \in \TT_{p,\fm'}$
is not in $\TT_{p,\fm'}'$, i.e.\ that $\TT_{p,\fm'}/\TT_{p,\fm'}' \neq 0$ and,
equivalently,
$S_p(\Gamma_1(N);\FF_p)_{\Katz,\fm'}^{\TT_{p,\fm'}'=0} = \Hom_{\FF_p}(\TT_{p,\fm'},\FF_p)^{\TT_{p,\fm'}'=0} \neq 0$.
This, however, is equivalent to the existence of a cusp form $f \in S_p(\Gamma_1(N);\Fbar_p)_\Katz$
such that $a_n(f)=0$ for all $p\nmid n$ and such that it is an eigenfunction for all
$T_n$ with $p \nmid n$ with eigenvalues corresponding to~$\fm'$.
By the theorem of \cite{katz-modp} used already in the proof
of Lemma~\ref{lem:ideal}, any such is of the
form $f = Fh$ with $h \in S_1(\Gamma_1(N);\Fbar_p)_\Katz$. Hence, there is a normalised
eigenform $g \in S_1(\Gamma_1(N);\Fbar_p)_\Katz$ such that
the $a_n(g)$ are the eigenvalues of $T_n$ on~$f$ for $p \nmid n$, whence~(\ref{lem:max:a:ii}).

Conversely, the existence of~$g$ implies that
$S_p(\Gamma_1(N);\FF_p)_{\Katz,\fm'}^{\TT_{p,\fm'}'=0} \neq 0$,
as it contains~$Fg$, so that $\TT_ {p, \fm'}' \neq \TT_{p,\fm'}$.

(\ref{lem:max:b})
The product decomposition into its localisations is a general fact of Artinian rings.
If $\TT_{p,\fm'}' = \TT_{p,\fm'}$, there is nothing to show. So we assume now that
this equality does not hold.
From Proposition~\ref{prop:ideal}~(\ref{prop:ideal:c}) we have the surjection of rings
$$ \TT_{p,\fm'}'[X]/(X^2-TX+D) \xrightarrow {X \mapsto U_p} \TT_{p,\fm'}.$$
Taking it mod~$\fm'$ yields $\FF[X]/(X^2-\bar{T}X+\bar{D})$ on the left hand side with
$\FF = \TT_p'/\fm'$, which has at most two local factors, depending on whether the
quadratic equation has two distinct roots or a double one. Thus there can at most be
two local factors on the right hand side.
Modulo~$\fm'$, the quadratic polynomial is in fact $X^2 - a_p(g)X + \epsilon(p)$,
which follows from the explicit shape of~$U_p$ given in Equation~\eqref{eq:hecke};
see also Corollary~\ref{cor:ideal}~(\ref{cor:ideal:a}).
The ordinarity is now also clear since $\epsilon(p)$ is non-zero in~$\FF$.
\end{proof}

We remark that it can happen that $a_p(g)^2 \neq 4 \epsilon(p)$
(this is the so-called {\em $p$-distinguished} case), but that nevertheless the algebra
$\TT_{p,\fm'}$ is local because the distinct roots of $X^2 - a_p(g)X + \epsilon(p)$
might only lie in a quadratic extension of~$\FF$.
(In a previous version of this article we had referred to the case
when $\TT_{p,\fm'}$ is non-local as `$p$-distinguished', which was very misleading.)

We assume henceforth that $\fm'$ comes from weight~$1$ and is hence ordinary.
We write $\Psi_{\fm'}$ for the localisation of $\Psi$ (from Equation~\eqref{eq:psi}) at~$\fm'$ and similarly
$I_{\fm'} = \TT_{p,\fm'}' \cap U_p \TT_{p,\fm'}'$.

\begin{cor}\label{cor:ideal}
Let $\fm'$ be a maximal ideal of $\TT_p'$ which comes from weight~$1$.
\begin{enumerate}[(a)]
\item\label{cor:ideal:a}
We have
$\big(S_1(\Gamma_1(N);\FF_p)_{\Katz,\fm'}\big)^{\oplus 2}
\overset{\Psi_{\fm'}}{\cong}  S_p(\Gamma_1(N);\FF_p)_{\Katz,\fm'}^{I_{\fm'}=0}$.
The operator $U_p$ acts on the left hand side as $\mat {T_p}{1}{-\diam p} 0$.
The operator $T_n$ for $p \nmid n$ acts as $\mat {T_n}00{T_n}$.

\item\label{cor:ideal:b}
There is a natural isomorphism $\TT_{p,\fm'}/I_{\fm'} \cong \TT_{1,\fm'} \oplus \TT_{1,\fm'}$
of $\TT_{p,\fm'}'$-modules.
The operator $U_p$ acts on the right hand side as $\mat {T_p}{-\diam p}1 0$.
The operator $T_n$ for $p \nmid n$ acts as $\mat {T_n}00{T_n}$.

\item\label{cor:ideal:c}
The ring homomorphism
$\TT_{p,\fm'}' \xrightarrow{T_n \mapsto T_n} \TT_{1,\fm'}$ is surjective with kernel~$I_{\fm'}$.
\end{enumerate}
\end{cor}

\begin{proof}
(\ref{cor:ideal:a}) By Lemma~\ref{lem:max}, $U_p$ is invertible and
by Proposition~\ref{prop:ideal}~(\ref{prop:ideal:d}), $I_{\fm'}$ is an ideal of~$\TT_{p,\fm'}$.
Thus, we have $I_{\fm'} = U_p^{-1} I_{\fm'} = \TT_{p,\fm'}' \cap U_p^{-1} \TT_{p,\fm'}'$.
Since by Proposition~\ref{prop:ideal}~(\ref{prop:ideal:b})
$\TT_{p,\fm'}'+U_p\TT_{p,\fm'}'=\TT_{p,\fm'}$, we have the natural isomorphism
$\TT_{p,\fm'}/I_{\fm'} \cong \TT_{p,\fm'}/\TT_{p,\fm'}' \oplus \TT_{p,\fm'}/(U_p\TT_{p,\fm'}')$
of $\TT'$-modules. It follows that
\begin{align*}
 \Hom_{\FF_p}(\TT_{p,\fm'},\FF)^{I_{\fm'}=0}
& \cong   \Hom_{\FF_p}(\TT_{p,\fm'},\FF)^{\TT_{p,\fm'}'=0} \oplus \Hom_{\FF_p}(\TT_{p,\fm'},\FF)^{U_p^{-1}\TT_{p,\fm'}'=0}\\
& = \Hom_{\FF_p}(\TT_{p,\fm'},\FF)^{\TT_{p,\fm'}'=0} \oplus
 U_p\Hom_{\FF_p}(\TT_{p,\fm'},\FF)^{\TT_{p,\fm'}'=0} \\
& \cong A(S_1(\Gamma_1(N))_\Katz)_{\fm'} \oplus F(S_1(\Gamma_1(N))_\Katz)_{\fm'}
 = \Image(\Psi_{\fm'}),
\end{align*}
using Lemma~\ref{lem:ideal}. Moreover, Equation~\eqref{eq:katz} gives an isomorphism
$$S_p(\Gamma_1(N);\FF_p)_{\Katz,\fm'}^{I_{\fm'}=0} \cong \Hom_{\FF_p}(\TT_{p,\fm'},\FF)^{I_{\fm'}=0}.$$
The shapes of $U_p$ and $T_n$ are taken from Equation~\eqref{eq:hecke}.

(\ref{cor:ideal:b}) Using Equation~\eqref{eq:katz}, (\ref{cor:ideal:a}) can be reformulated
as an isomorphism
$$ \Hom_{\FF_p}(\TT_{p,\fm'}/I_{\fm'},\FF_p) \cong \Hom_{\FF_p}(\TT_{p,\fm'},\FF_p)^{I_{\fm'}=0}
\cong \Hom_{\FF_p}(\TT_{1,\fm'},\FF_p)^{\oplus 2}.$$
Dualising it yields the statement, and the matrices are just the transposes of the matrices
in the previous part.

(\ref{cor:ideal:c}) The algebra generated by the $T_n$ with $p \nmid n$
on the left hand side of (\ref{cor:ideal:b})
is $\TT_{p,\fm'}'/I_{\fm'}$ and on the right hand side $\TT_{1,\fm'}'$,
which equals $\TT_{1,\fm'}$ by Proposition~\ref{prop:ideal}~(\ref{prop:ideal:a}).
\end{proof}

We refer to (\ref{cor:ideal:b}) as `doubling of Hecke algebras'.
Part (\ref{cor:ideal:c}) is the key for the definition
of the Galois representation with coefficients in~$\TT_{1,\fm'}$.

\subsection*{Passage to weight two}

In order to work on the Jacobian of a modular curve,
we pass from weight~$p$ to weight~$2$, which is only necessary if $p > 2$.
We assume this for this subsection.

\begin{prop}\label{prop:two}
Let $N \ge 5$, $p \nmid N$, $p > 2$ and $\tilde{\fm}$ be an ordinary maximal ideal of
the Hecke algebra $\TT_p(\Gamma_1(N);\FF_p)_\Katz$.
Then there is a unique maximal ideal~$\fm_2$ of $\FF_p \otimes_\ZZ \TT_2(\Gamma_1(Np))_\cl$
such that $T_n \mapsto T_n$ for all~$n$ defines a ring isomorphism
$\TT_p(\Gamma_1(N);\FF_p)_{\Katz,\tilde{\fm}} \cong (\FF_p \otimes_\ZZ \TT_2(\Gamma_1(Np))_\cl)_{\fm_2}$.
\end{prop}

\begin{proof}
This is due to Hida and follows, for instance,
from combining Proposition~\ref{prop:comparison} and \cite{nongorenstein}, Proposition~2.3.
\end{proof}

Remembering $\TT_{p,\fm'}= \prod_{i=1}^n \TT_{p,\tilde{\fm}_i}$ (see Lemma~\ref{lem:max}),
we obtain that after localisation at ordinary $\fm'$, the Hecke algebra $\TT_{p,\fm'}$
acts on the $p$-torsion of the Jacobian of $X_1(Np)$.
We shall henceforth use this action without specifying
the isomorphism from Proposition~\ref{prop:two} explicitly.

\section{The Galois representation of weight one}\label{sec:galois}

In this section we shall construct the Galois representation~$\rho_\fm$, identify it on
the Jacobian of a suitable modular curve and derive that it `doubles' from the `doubling
of Hecke algebras'.

We collect some statements and pieces of notation which are in place for the whole of this
section.

\begin{notat}\label{notat:galois}
Next to Notation~\ref{notat:mf} we use the following pieces of notation and the following assumptions.
\begin{itemize}
\item $\TT_p = \TT_p(\Gamma_1(N);\FF_p)_\Katz$ denotes the full Hecke algebra on $S_p(\Gamma_1(N);\FF_p)_\Katz$
and $\TT_p'$ is its subalgebra generated by those $T_n$ with $p \nmid n$. The $p$-th Hecke operator
is denoted~$U_p$.

\item $\TT_1$ is the Hecke algebra on $S_1(\Gamma_1(N);\FF_p)_\Katz$ and it is equal to $\TT_1'$
(see Proposition~\ref{prop:ideal}).
The $p$-th Hecke operator is denoted~$T_p$.

\item The map $\TT_p' \xrightarrow{T_n \mapsto T_n} \TT_1$ defines a ring surjection with kernel
$I = \TT_p' \cap U_p \TT_p'$ (see Corollary~\ref{cor:ideal}~(\ref{cor:ideal:c})).

\item Let $\fm'$ be a maximal ideal of $\TT_p'$ which comes from weight~$1$ and corresponds
to a normalised eigenform $g \in S_1(\Gamma_1(N);\Fbar_p)_\Katz$ (see Lemma~\ref{lem:max}).
Let $\epsilon$ be the Dirichlet character of~$g$. Then $\fm'$ is ordinary (see Lemma~\ref{lem:max}).
Denote by $\fm$ the maximal ideal of $\TT_1$ the preimage of which in $\TT_p'$ is~$\fm'$,
whence $\TT_p'/\fm' = \TT_1/\fm$.
Then $\fm$ corresponds to the $\Gal(\Fbar_p/\FF_p)$-conjugacy class of~$g$, i.e.\ it is
the kernel of the ring homomorphism $\TT_1 \xrightarrow{T_n \mapsto a_n(g)} \Fbar_p$.

\item Either $\TT_{p,\fm'} \cong \TT_{p,\tilde{\fm}_1} \times \TT_{p,\tilde{\fm}_2}$
(the {\em non-local case}), or $\TT_{p,\fm'}$ is local (see Lemma~\ref{lem:max}).

\end{itemize}
\end{notat}

\subsection*{Existence}

By work of Shimura and Deligne there is a Galois representation
$$\rhobar_\fm= \rhobar_{\fm'}: \Gal(\Qbar/\QQ) \to \GL_2(\TT_1/\fm) = \GL_2(\TT_p'/\fm')$$
characterised by the property that it is unramified outside $Np$ and
$$\charpoly(\rhobar_\fm)(\Frob_\ell) = X^2 - T_\ell X + \diam{\ell} \in \TT_1/\fm[X]\cong\TT_p'/\fm'[X]$$
for all primes $\ell \nmid Np$.
Under the assumption that $\rhobar_\fm$ is absolutely irreducible, Carayol in
\cite{carayol-anneau}, Th\'eor\`eme~3,
shows the existence of a Galois representation
$$\rho_{\fm'}: \Gal(\Qbar/\QQ) \to \GL_2(\TT_{p,\fm'}')$$
characterised by the property that it is unramified outside $Np$ and
$\charpoly(\rho_{\fm'}(\Frob_\ell)) = X^2 - T_\ell X + \diam{\ell} \in \TT_{p,\fm'}'[X]$
for all primes $\ell \nmid Np$.
In fact, the reference gives a twist of this representation.
Later on, we are going to be more precise about which twist it is.
As a general rule, we denote $\rhobar$ a representation with coefficients
in a finite field or $\Fbar_p$ and $\rho$ when the coefficients are in a Hecke algebra.

\begin{prop}\label{prop:existence}
Let $\fm$ be a maximal ideal of $\TT_1$ such that $\rhobar_\fm$ is absolutely
irreducible. Then there is a Galois representation
$$\rho_\fm: \Gal(\Qbar/\QQ) \to \GL_2(\TT_{1,\fm})$$
characterised by the property that it is unramified outside $Np$ and
$\charpoly(\rho_\fm(\Frob_\ell)) = X^2 - T_\ell X + \diam{\ell} \in \TT_{1,\fm}[X]$
for all primes $\ell \nmid Np$.
\end{prop}

\begin{proof}
It suffices to compose $\rho_{\fm'}$ with $\GL_2(\TT_{p,\fm'}') \to \GL_2(\TT_{1,\fm})$
coming from Corollary~\ref{cor:ideal}~(\ref{cor:ideal:c}).
\end{proof}

\subsection*{The $p$-divisible group for $p=2$}

Assume for the moment that $p=2$.
Let $J$ be the Jacobian $J_1(N)$ of $X_1(N)$, which is defined over~$\QQ$.
Let $G$ be the $\fm'$-component of the $p$-divisible group $J[p^\infty]_\QQ$ attached to~$J$.

A word of explanation is necessary (see also \cite{gross}, Section~12).
The maximal ideals $\tilde{\fm}$ of $\TT_p$
containing $\fm'$ correspond under pull-back to unique maximal ideals of the Hecke algebra
$\ZZ_p \otimes_\ZZ \TT_p(\Gamma_1(N))_\cl$, using Proposition~\ref{prop:comparison}.
This Hecke algebra acts on the Tate module of~$J$ and localisation at each~$\tilde{\fm}$
gives a direct factor of it.
Then $G$ is the direct product of the (at most two  by Lemma~\ref{lem:max})
corresponding $p$-divisible groups.
If $\TT_{p,\fm'}$ is non-local, then we shall denote by $G_1$ and $G_2$ the two
$p$-divisible groups such that $G = G_1 \times G_2$.

\subsection*{The $p$-divisible group for $p>2$}

Assume now $p>2$. We proceed very similarly to the above:
Let $J$ be the Jacobian $J_1(Np)$ of $X_1(Np)$, which is defined over~$\QQ$.
Let $G$ be the $\fm'$-component of the $p$-divisible group $J[p^\infty]_\QQ$ attached to~$J$.

Here we use that the ideals $\tilde{\fm}$ of $\TT_p$ containing $\fm'$ correspond
to unique maximal ideals of the Hecke algebra
 $\FF_p \otimes_\ZZ \TT_2(\Gamma_1(Np))_\cl$ by Proposition~\ref{prop:two}.
In turn they give rise, by taking preimages, to unique maximal ideals of
$\ZZ_p \otimes_\ZZ \TT_2(\Gamma_1(Np))_\cl$.
For each of these (at most two, by Lemma~\ref{lem:max}) maximal ideals we take the
$p$-divisible group of the corresponding factor of the Tate module of~$J$.
Thus, if $\TT_{p,\fm'}$ is non-local, then $G$ is of the form $G_1 \times G_2$.

\subsection*{Properties of the $p$-divisible group}

We assume that $G$ (and $G_1$ and $G_2$ in the non-local case)
is as defined above (for either $p=2$ or $p>2$).

\begin{prop}\label{prop:mult}
The $p$-divisible group~$G$ acquires good reduction over~$\ZZ_p[\zeta_p]$.
Let $G^0$ and $G^e$ be the connected component and the \'etale quotient of~$G$
over~$\ZZ_p[\zeta_p]$, respectively.
\begin{enumerate}[(a)]
\item\label{prop:mult:a}
The module $G^0[p](\Qbar_p)$ is unramified over~$\QQ_p(\zeta_p)$ and
there is an isomorphism $G^0[p](\Qbar_p) \cong \TT_{p,\fm'}$ of $\TT_{p,\fm'}$-modules,
under which any arithmetic Frobenius $\Frob_p \in \Gal(\Qbar_p/\QQ_p(\zeta_p))$ acts as $U_p^{-1}$.

\item\label{prop:mult:b}
The exact sequence $0 \to G^0 \to G \to G^e \to 0$ gives rise to the exact sequence
$$ 0 \to \TT_{p,\fm'} \to G[p](\Qbar_p) \to \Hom_{\FF_p}(\TT_{p,\fm'},\FF_p) \to 0$$
of $\TT_{p,\fm'}$-modules, under the identification of (\ref{prop:mult:a}) and its dual.

\end{enumerate}
\end{prop}

\begin{proof}
This follows immediately from applying \cite{multiplicities}, Proposition~2.2, Corollary~2.3
and Theorem~3.1 for all maximal ideals~$\tilde{\fm} \subset \TT_p$ containing~$\fm'$.
We stress that those results were all derived from~\cite{gross}.
\end{proof}

Since in this article we are using arithmetic Frobenius elements, and on modular curves
(with level structure of the type $\mu_N \hookrightarrow E[N]$) geometric ones are more natural,
we have to twist our representations at various places.

It is well-known that $\rhobar_{\fm} \otimes \epsilon^{-1}$ is contained in the $\fm'$-kernel
$G[p](\Qbar)[\fm']$ of  $G[p](\Qbar)$ (possibly more than once, see e.g.\ \cite{multiplicities}, Proposition~4.1).

The following theorem is the result of the work of many authors. We do not intend to give all
the original references, but, content ourselves by referring to a place in the literature
where the statements appear as we need them.

\begin{thm}[Mazur, Wiles, Gross, Ribet, Buzzard, Tilouine, Edixhoven, W.]\label{thm:mult}
Assume that $\rhobar_\fm$ is absolutely irreducible.
Then the following statements are equivalent:
\begin{enumerate}[(i)]
\item\label{thm:mult:i}
$\rhobar_\fm$ is unramified at~$p$ and $\rhobar_\fm(\Frob_p)$ is non-scalar or $\rhobar_\fm$ is ramified at~$p$.
\item\label{thm:mult:ii}
$G[p](\Qbar)[\tilde{\fm}] \cong \rhobar_\fm$ for any maximal ideal $\tilde{\fm} \subset \TT_p$
containing~$\fm'$, i.e.\ $\rhobar_\fm$ satisfies {\em multiplicity one on the Jacobian}.
\item\label{thm:mult:iii}
$\TT_{p,\fm'} \cong \Hom_{\FF_p}(\TT_{p,\fm'} ,\FF_p)$,
i.e.\ $\TT_{p,\fm'}$ is {\em Gorenstein}.
\end{enumerate}

If the equivalent statements hold, then
$G[p](\Qbar) \cong \TT_{p,\fm'} \oplus \TT_{p,\fm'}$.
\end{thm}

\begin{proof}
For the implication (\ref{thm:mult:i}) $\Rightarrow$ (\ref{thm:mult:ii}) we refer, for instance,
to \cite{nongorenstein}, Theorem~1.2.
The implication (\ref{thm:mult:ii}) $\Rightarrow$ (\ref{thm:mult:i}) is the content of \cite{multiplicities}, Corollary~4.5.
The equivalence of (\ref{thm:mult:ii}) and (\ref{thm:mult:iii}) is proved, for instance,
in \cite{nongorenstein}, Proposition~2.2~(b).
That (\ref{thm:mult:ii}) implies the final statement is, for example,
proved in \cite{nongorenstein}, Proposition~2.1.
\end{proof}

Note that by Proposition~\ref{prop:weightone} the case that $\rhobar_\fm$ is ramified at~$p$
is known not to occur in almost all cases. We are proving in Corollary~\ref{cor:main} that it
actually never occurs.

\subsection*{The Galois representation on the Jacobian}

We proceed under the following assumptions:

\begin{ass}\label{ass:main}
We continue to use Notation \ref{notat:mf} and~\ref{notat:galois}. Moreover:
\begin{itemize}
\item $\rhobar_\fm$ denotes the residual Galois representation introduced above. It
is the residual representation attached to~$g$.
We assume that $\rhobar_\fm$ is absolutely irreducible.

\item Let $G$ (and $G_1$, $G_2$ in the non-local case) be the $p$-divisible
group introduced earlier.

\item We assume that $\rhobar_\fm$ satisfies multiplicity one on the Jacobian so that
$G[p](\Qbar) \cong \TT_{p,\fm'} \oplus \TT_{p,\fm'}$ (see Theorem~\ref{thm:mult}).

\item Let $\tilde{\epsilon}: \Gal(\Qbar/\QQ) \twoheadrightarrow \Gal(\QQ(\zeta_N)/\QQ) \cong \ZZ/(N)^\times
\xrightarrow{a \mapsto \diam a} \TT_p'^\times$.
Note that composing $\tilde{\epsilon}$
with $\TT_{p,\fm'}'^\times \to (\TT_p'/\fm')^\times$ equals~$\epsilon$, the Dirichlet character of~$g$,
seen as a character of $\Gal(\Qbar/\QQ)$.

\end{itemize}
\end{ass}

The next proposition can be considered as a `doubling of Galois representations'.

\begin{prop}\label{prop:jac}
We use Assumption~\ref{ass:main}.
Let $\rho_{I_{\fm'}} := G[p](\Qbar) \otimes_{\TT_{p,\fm'}'} \TT_{1,\fm}$.
Then there is an isomorphism
$$ \rho_{I_{\fm'}} \cong \big( \rho_\fm \otimes \tilde{\epsilon}^{-1}\big)
                        \oplus \big( \rho_\fm \otimes \tilde{\epsilon}^{-1}\big).$$
of $\TT_{1,\fm}[\Gal(\Qbar/\QQ)]$-modules.
\end{prop}

\begin{proof}
From \cite{carayol-anneau}, 3.3.2, and Theorem~\ref{thm:mult} it follows that
$H := G[p](\Qbar) \cong \TT_{p,\fm'} \oplus \TT_{p,\fm'}$
as $\TT_{p,\fm'}[\Gal(\Qbar/\QQ)]$-modules and that it is
characterised by the property that it is unramified outside $Np$
and that the characteristic polynomial of $\Frob_\ell$
is $X^2 - T_\ell/\diam \ell X + 1/\diam \ell \in \TT_{p,\fm'}'[X]$
for all primes $\ell \nmid Np$. We recall that \cite{carayol-anneau} works with geometric
Frobenius elements, whereas we are using arithmetic ones, accounting for the differences
in the formulae.

By Th\'eor\`eme~2 of loc.~cit., $H$ is obtained by scalar extension of
some $\TT_{p,\fm'}'[\Gal(\Qbar/\QQ)]$-module $H' \cong \TT_{p,\fm'}' \oplus \TT_{p,\fm'}'$,
i.e.\ $H = H' \otimes_{\TT_{p,\fm'}'} \TT_{p,\fm'}$.
Note that in this description $H$ is a $\Gal(\Qbar/\QQ)$-module via an action
on $H'$ and a $\TT_{p,\fm'}$-module via the natural action on $\TT_{p,\fm'}$ in the tensor product.

Next we have the following isomorphisms of $\TT_{p,\fm'}'[\Gal(\Qbar/\QQ)]$-modules:
\begin{multline*}
       H \otimes_{\TT_{p,\fm'}} \TT_{p,\fm'}/I_{\fm'}
\cong \big(H' \otimes_{\TT_{p,\fm'}'} \TT_{p,\fm'}\big)
         \otimes_{\TT_{p,\fm'}} \TT_{p,\fm'}/I_{\fm'}
\cong  H' \otimes_{\TT_{p,\fm'}'} \TT_{p,\fm'}/I_{\fm'}\\
\cong  H' \otimes_{\TT_{p,\fm'}'} \big(\TT_{1,\fm} \oplus \TT_{1,\fm}\big)
\cong \big(H' \otimes_{\TT_{p,\fm'}'} \TT_{1,\fm}\big)
      \oplus \big(H' \otimes_{\TT_{p,\fm'}'} \TT_{1,\fm}\big),
\end{multline*}
where we used Corollary~\ref{cor:ideal}~(\ref{cor:ideal:b}).
Note that the $\TT_{p,\fm'}'[\Gal(\Qbar/\QQ)]$-action factors through to
give a $\TT_{1,\fm}[\Gal(\Qbar/\QQ)]$-action.

Recall that
$\rho_\fm: \Gal(\Qbar/\QQ) \to \GL_2(\TT_{p,\fm'}'/I_{\fm'}) \cong \GL_2(\TT_{1,\fm})$
is characterised by it being unramified outside $Np$ and the characteristic polynomial
of $\Frob_\ell$ for $\ell \nmid Np$ being equal to $X^2 - T_\ell X + \diam \ell$.
Hence, the characteristic polynomial of $(\rho_\fm \otimes \tilde{\epsilon}^{-1})(\Frob_\ell)$
is $X^2 - T_\ell/\diam \ell X + 1/\diam \ell$.
Since $H' \otimes_{\TT_{p,\fm'}'} \TT_{1,\fm}$ satisfies the same properties,
it agrees with~$\rho_\fm \otimes \tilde{\epsilon}^{-1}$ by \cite{carayol-anneau}, Th\'eor\`eme~1.
\end{proof}

\section{Proof of Theorem~\ref{thm:main}}\label{sec:proof}

In this section we prove Theorem~\ref{thm:main}.
We deal with the cases when $\TT_{p,\fm'}$ is local or not separately.

Let us first remark that $N \ge 5$ can be assumed without loss of generality as follows.
One can increase the level at some unramified auxiliary prime~$q\ge 5$, $q \neq p$ and apply
the theorem in level~$Nq$, yielding a Galois representation with coefficients in the weight~$1$
Hecke algebra for $\Gamma_1(Nq)$ which is unramified outside~$Nq$.
Since the Hecke algebra for $\Gamma_1(N)$ is a quotient of the one for~$\Gamma_1(Nq)$,
one obtains the desired Galois representation, which is hence also unramified outside $Nq$.
Choosing a different auxiliary~$q$, one sees that
the Galois representation for $\Gamma_1(N)$ is unramified at the auxiliary prime.

\subsection*{No tame ramification}

We first show that there cannot be any tame ramification.

\begin{lem}\label{lem:tame}
Let $\TT$ be a finite dimensional local $\FF$-algebra with maximal
ideal~$\fm$ for a finite extension $\FF/\FF_p$.
If $A \in \ker(\GL_n(\TT) \to \GL_n(\TT/\fm))$
is a matrix such that $A^{p-1}=1$, then $A=1$.
\end{lem}

\begin{proof}
There is a matrix~$M$ all of whose entries are in~$\fm$ such that $A=1+M$.
Thus $A=A^{p^r} = (1+M)^{p^r} = 1+ M^{p^r}$ for all $r \in \NN$.
As $\fm$ is a nilpotent ideal and all entries of $M^{p^r}$ lie in $\fm^{p^r}$,
it follows that $M=0$.
\end{proof}

\begin{prop}\label{prop:tame}
Let $\TT$ be a finite dimensional local $\FF$-algebra with maximal
ideal~$\fm$ for a finite extension $\FF/\FF_p$.
Let $C$ be a subgroup of $\ZZ/(p-1)$ and $\rho: C \to \GL_n(\TT)$ a representation such that
the residual representation $\rhobar: C \to \GL_n(\TT/\fm)$ is trivial.
Then $\rho$ is trivial.
\end{prop}

\begin{proof}
As $\rhobar$ is trivial, $\rho$ takes its values in $\ker(\GL_n(\TT) \to \GL_n(\TT/\fm))$.
But, this group does not have any nontrivial element of order dividing~$p-1$
by Lemma~\ref{lem:tame}, whence $\rho$ is the trivial representation.
\end{proof}

\begin{cor}\label{cor:tame}
Let $\TT$ be a finite dimensional local $\FF$-algebra with maximal
ideal~$\fm$ for a finite extension $\FF/\FF_p$.
Let $\rho: \Gal(\Qbar_p/\QQ_p) \to \GL_2(\TT)$ be a representation and let
$\rhobar: \Gal(\Qbar_p/\QQ_p) \to \GL_2(\TT/\fm)$ be its residual representation.
Assume that the semisimplification of $\rhobar$ is unramified
and that the restriction of $\rho$ to $\Gal(\Qbar_p/\QQ_p(\zeta_p))$ is unramified.

Then $\rho$ is unramified.
\end{cor}

\begin{proof}
As the semisimplification of $\rhobar$ is unramified and the restriction of $\rho$
to $\Gal(\Qbar_p/\QQ_p(\zeta_p))$ is also unramified, it follows that $\rhobar$ is unramified.
Moreover, the image of the inertia group has to be a subgroup of $\ZZ/(p-1)$,
whence it acts trivially by Proposition~\ref{prop:tame}.
\end{proof}

\subsection*{The non-local case}

\begin{proof}[Proof of Theorem~\ref{thm:main} -- the non-local case]
Note that the assumptions of Theorem~\ref{thm:main} imply that
Assumption~\ref{ass:main} is satisfied due to Theorem~\ref{thm:mult}.
We now assume that we are in the non-local case, i.e.\
$\TT_{p,\fm'} \cong \TT_{p,\tilde{\fm}_1} \times \TT_{p,\tilde{\fm}_2}$.
Let $m_i \in \FF_p[X]$ be the minimal polynomial of~$U_p^{-1}$
acting on $\TT_{p,\tilde{\fm}_i}$. Then $m_1$ and $m_2$ are powers of coprime irreducible polynomials.
We obtain
$$ G[p](\Qbar)/I_{\fm'} = G_1[p](\Qbar)/I_{\fm'} \oplus G_2[p](\Qbar)/I_{\fm'}$$
and $G_i[p](\Qbar)/I_{\fm'}$ is characterised by the fact that $m_i(U_p^{-1})$ annihilates it.
From Proposition~\ref{prop:jac} it follows that $G[p](\Qbar)/I_{\fm'}$
is isomorphic to
$\big(\rho_\fm \otimes \tilde{\epsilon}^{-1}\big)
\oplus \big(\rho_\fm \otimes \tilde{\epsilon}^{-1}\big)$
as $\TT_{1,\fm}[\Gal(\Qbar/\QQ)]$-modules.
But, as such $G_1[p](\Qbar)/I_{\fm'} \cong G_2[p](\Qbar)/I_{\fm'}$, thus for both $i=1,2$ we have
$\rho_\fm \otimes \tilde{\epsilon}^{-1} \cong G_i[p](\Qbar)/I_{\fm'}$.

We are now going to work locally and let $\cG = \Gal(\Qbar_p/\QQ_p(\zeta_p))$ and $\cI$
its inertia group.
By Proposition~\ref{prop:mult}~(\ref{prop:mult:a}) applied to~$G_i$ we obtain for $i=1,2$ that
$$ G_i^0[p](\Qbar_p)/I_{\fm'} \hookrightarrow \big(G_i[p](\Qbar_p)/I_{\fm'}\big)^{\cI}
\cong \big(\rho_\fm \otimes \tilde{\epsilon}^{-1}\big)^{\cI}$$
and that $\Frob_p$ on the left hand side acts through~$U_p^{-1}$, whence the image
of the map is annihilated by $m_i(\Frob_p)$.
As the polynomials $m_1$ and $m_2$ are coprime,
$G^0[p](\Qbar_p)/I_{\fm'} \cong G_1^0[p](\Qbar_p)/I_{\fm'} \oplus G_2^0[p](\Qbar_p)/I_{\fm'}$
is a subrepresentation of $\big(\rho_\fm \otimes \tilde{\epsilon}^{-1}\big)^{\cI}$.
Counting $\FF_p$-dimensions, it follows that
$$G^0[p](\Qbar_p)/I_{\fm'} \cong \big(\rho_\fm \otimes \tilde{\epsilon}^{-1}\big)^{\cI}
\cong \rho_\fm \otimes \tilde{\epsilon}^{-1}.$$
Consequently, $\rho_\fm$ is unramified at~$p$, using Corollary~\ref{cor:tame} and
taking into account that the semisimplification of $\rhobar_\fm$ restricted
to $\Gal(\Qbar_p/\QQ_p)$ is unramified at~$p$.
Moreover, again due to Proposition~\ref{prop:mult}~(\ref{prop:mult:a}) the characteristic
polynomial of $\Frob_p$ on~$\rho_\fm \otimes \tilde{\epsilon}^{-1}$
is the one of $U_p^{-1}$, which is $X^2 - T_p/\diam p X + 1/\diam p$ (see Corollary~\ref{cor:ideal}),
so that the one of $\rho_\fm(\Frob_p)$ is as claimed.
\end{proof}

\subsection*{The local case}

\begin{prop}\label{prop:indec}
Let $R$ be a local $\FF_p$-algebra which is finite dimensional as an $\FF_p$-vector
space and let $\fm$ be its maximal ideal. Put $\FF = R/\fm$. Let $\cG$ be a group.
Let $M,N,Q$ be $R[\cG]$-modules which are free of rank~$2$ as $R$-modules and suppose
we have an exact sequence
$$ 0 \to N \xrightarrow{\alpha} M\oplus M \to Q \to 0$$
of $R[\cG]$-modules.
Suppose further that $\overline{N} := R/\fm \otimes_R N$ is indecomposable
as an $\FF[\cG]$-module.

Then $N \cong M \cong Q$ as $R[\cG]$-modules.
\end{prop}

\begin{proof}
Write $\overline{M} := R/\fm \otimes_R M$. Counting dimensions as $\FF$-vector spaces
it follows that the sequence
$$ 0 \to \overline{N} \xrightarrow{\overline{\alpha}}
\overline{M}\oplus \overline{M} \to \overline{Q} \to 0$$
is an exact sequence of $\FF[\cG]$-modules.
Consider the composite map
$$ \phi_i: N \overset{\alpha}{\hookrightarrow} M \oplus M \xrightarrow{\pr_i} M$$
for $i=1,2$, where $\pr_i$ is the projection on the $i$-th summand. Note that the
$\phi_i$ are homomorphisms of $R[\cG]$-modules.
Tensor $\phi_i$ with $R/\fm$ to obtain
$$ \overline{\phi_i}: \overline{N} \overset{\overline{\alpha}}{\hookrightarrow}
\overline{M} \oplus \overline{M} \xrightarrow{\overline{\pr_i}} \overline{M}.$$
Note that the cases $\dim_\FF \Image(\overline{\phi_i}) \le 1$ for $i=1,2$ cannot
occur: If one of the dimensions is $1$ and the other $0$ or if both are~$0$,
then one has a contradiction to the injectivity of~$\overline{\alpha}$. If both are~$1$,
then $\overline{N} \cong \Image(\overline{\phi_1}) \oplus \Image(\overline{\phi_2})$
as $\FF[\cG]$-modules, which contradicts the assumed indecomposability of~$\overline{N}$.
Hence, there is $i\in \{1,2\}$ such that $\dim_\FF \Image(\overline{\phi_i})=2$.
Hence, $\overline{\phi_i}$ is an isomorphism $\overline{N} \to \overline{M}$.
It follows that $\phi_i:N \to M$ is surjective.
Indeed, tensoring the exact sequence $N \xrightarrow{\phi_i} M \to S \to 0$ over~$R$ with $R/\fm$,
shows that $\overline{S} = R/\fm \otimes_R S = 0$, whence $S = 0$ by Nakayama's lemma.
As $N$ and $M$ are finite sets, $\phi_i$ is an isomorphism of $R[\cG]$-modules.
\end{proof}

\begin{proof}[Proof of Theorem~\ref{thm:main} -- the local case]
Note that the assumptions of Theorem~\ref{thm:main} imply that
Assumption~\ref{ass:main} is satisfied due to Theorem~\ref{thm:mult}.
We now assume that $\TT_{p,\fm'}$ is local.
We are going to deduce the statement from Proposition~\ref{prop:indec}.
For $R$ we take $\TT_{1,\fm}$ and we let $\cG := \Gal(\Qbar_p/\QQ_p(\zeta_p))$.
In Proposition~\ref{prop:jac} we have seen that $\rho_{I_{\fm'}}$
is isomorphic to
$\big(\rho_\fm \otimes \tilde{\epsilon}^{-1}\big)
\oplus \big(\rho_\fm \otimes \tilde{\epsilon}^{-1}\big)$ as $R[\cG]$-modules
and we take $M$ to be $\rho_\fm \otimes \tilde{\epsilon}^{-1}$.

Next we reduce the exact sequence of Proposition~\ref{prop:mult}~(\ref{prop:mult:b})
modulo~$I_{\fm'}$. Due to multiplicity one, it remains exact (since it is split as
a sequence of $\TT_{p,\fm'}$-modules),
whence we obtain an exact sequence of $R[G]$-modules
$$ 0 \to \tilde{\rho} \to \rho_{I_{\fm'}} \to G^e[p](\Qbar_p)/I_{\fm'} \to 0,$$
where $\tilde{\rho} = G^0[p](\Qbar_p)/I_{\fm'}$.
By Proposition~\ref{prop:mult}~(\ref{prop:mult:a}) we know that $\tilde{\rho}$
is unramified as a $\cG$-module and it is free of rank~$2$ over~$R$.
Moreover, any arithmetic Frobenius at~$p$ acts through multiplication by~$U_p^{-1}$.
Also $G^e[p](\Qbar_p)/I_{\fm'}$ is free of rank~$2$ as an $R$-module.
Taking $\tilde{\rho}$ modulo~$\fm$ we obtain an indecomposable $R/\fm[\cG]$-module, where
the indecomposability is due to the formula for~$U_p$ (see Corollary~\ref{cor:ideal}).
Hence, we take $N$ to be~$\tilde{\rho}$, restricted to~$\cG$.

Thus, from Proposition~\ref{prop:indec} we obtain
$\rho_\fm \otimes \tilde{\epsilon}^{-1} \cong \tilde{\rho}$
as $\TT_{1,\fm}[\Gal(\Qbar_p/\QQ_p(\zeta_p))]$-modules, and, in particular, that
$\rho_\fm$ is unramified at~$p$, using Corollary~\ref{cor:tame} and
taking into account that the semisimplification of $\rhobar_\fm$
restricted to $\Gal(\Qbar_p/\QQ_p)$ is unramified at~$p$.
Moreover, again due to Proposition~\ref{prop:mult}~(\ref{prop:mult:a}) the characteristic
polynomial of $\Frob_p$ on~$\rho_\fm \otimes \tilde{\epsilon}^{-1}$
is the one of $U_p^{-1}$, which is $X^2 - T_p/\diam p X + 1/\diam p$ (see Corollary~\ref{cor:ideal}),
so that the one of $\rho_\fm(\Frob_p)$ is as claimed.
\end{proof}

\section{Examples}\label{sec:examples}

We illustrate Theorem~\ref{thm:main} by two examples. They both appeared first in Mestre's
appendix to~\cite{edix-jussieu} and we work them out in our context.

Both examples are of the following shape. Let $\FF$ be a finite field of characteristic~$p$
and $\TT := \FF[\epsilon] := \FF[X]/(X^2)$. Then we have the split exact sequence of groups:
$$ 0 \to \Mat_2(\FF)^0 \xrightarrow{A \mapsto 1 + \epsilon A} \SL_2(\TT)
\xrightarrow{\epsilon \mapsto 0} \SL_2(\FF) \to 1,$$
where $\Mat_2(\FF)^0$ denotes the $2\times 2$-matrices of trace zero
(considered here as an abelian group with respect to addition), on which
$\SL_2(\FF)$ acts by conjugation (i.e.\ it is the adjoint representation).
If $p>2$, then this representation is irreducible, if $p=2$ it has non-trivial submodules.

\subsection*{Example $p=2$, $N=229$}

In this case there is only one normalised eigenform $g \in S_1(\Gamma_1(N);\Fbar_2)_\Katz$ and thus
only a unique maximal ideal $\fm \subset \TT_1$.
For example using {\sc Magma} (\cite{magma}) and a package developed by the author
(see Appendix~B of~\cite{edix-jussieu} for an old version),
one computes that $\TT_1 \cong \FF_2[\epsilon]$ and that
$\rhobar_\fm = \rhobar_g: \Gal(\Qbar/\QQ) \twoheadrightarrow \SL_2(\FF_2) \cong S_3$.
If $\ker(\rhobar_g) = \Gal(\Qbar/K)$, then $K$ is the Hilbert class field of~$\QQ(\sqrt{229})$.
Let us call $G$ the image of $\rho_\fm: \Gal(\Qbar/\QQ) \to \SL_2(\TT)$ coming
from Theorem~\ref{thm:main}. It turns out that
$G \cap \Mat_2(\FF_2)^0 = \{ \mat 0000, \mat 0110, \mat 1101, \mat 1011 \}$
(with the intersection being taken with respect to the map $A \mapsto 1+\epsilon A$)
and that $G \cong S_4$.
In fact, this example can be obtained by reducing $\rho_f$, where $f$ is a holomorphic weight~$1$
cuspidal eigenform with $\rho_f$ having projective image $S_4$.

Hence, the fact that $\rho_\fm$ is unramified at~$2$, which follows from Theorem~\ref{thm:main},
can already be deduced from the theorem of Deligne and Serre.

\subsection*{Example $p=2$, $N=1429$}

In this case there is a normalised eigenform $g \in S_1(\Gamma_1(N);\Fbar_2)_\Katz$ such that
the image of $\rhobar_g$ is isomorphic to $\SL_2(\FF_8)$. As $\SL_2(\FF_8)$ is not
a quotient of any finite subgroup of $\GL_2(\CC)$, this implies, as noted by Mestre,
that $g$ is not the reduction of any holomorphic weight~$1$ eigenform.
Let $\fm$ be the maximal ideal of~$\TT_1$ corresponding to~$g$.
One computes that $\TT_{1,\fm} \cong \FF_8[\epsilon]$.
If $\ker(\rhobar_g) = \Gal(\Qbar/K)$, then $K$ is a Galois extension of~$\QQ$
with Galois group $\SL_2(\FF_8)$ which is unramified outside the prime~$1429$.

Now consider $\rho_\fm:\Gal(\Qbar/\QQ) \to \SL_2(\TT_{1,\fm})$ from Theorem~\ref{thm:main}
and let $L$ be the Galois extension of~$\QQ$ such that $\Gal(\Qbar/L) \cong \ker(\rho_\fm)$,
which is unramified outside~$1429$.
After checking many Frobenius traces it seems very likely that $L$ is
$K(\sqrt{1429})$ and, hence, that $G := \Image(\rho_\fm) \cong \SL_2(\FF_8) \times \ZZ/(2)$.
Explicitly, $G \cap \Mat_2(\FF_2)^0 = \{ \mat 0000, \mat 1001 \}$.

In this case it is clear that $L$ is unramified at~$2$ without appealing to Theorem~\ref{thm:main}.
However, the remarkable phenomenon is that this extension $L/\QQ$ is detected by
weight one Katz forms through their Hecke algebras.
This points in the direction that one should ask if $\TT_{1,\fm}$ is in fact a universal
deformation ring of $\rhobar_\fm$
in the category of local $\FF_p$-algebras with residue field~$\TT_1/\fm$ for
the local conditions of being unramified at~$p$ and minimally ramified elsewhere.

\bibliographystyle{amsalpha}
\bibliography{GalRepWO}
\bigskip

\noindent
Gabor Wiese\\
Universit\'e du Luxembourg\\
Facult\'e des Sciences, de la Technologie et de la Communication\\
6,~rue Richard Coudenhove-Kalergi\\
L-1359 Luxembourg\\
Luxembourg\\
E-mail:   {\tt gabor.wiese@uni.lu}\\
Web page: {\tt http://math.uni.lu/$\sim$wiese}\\

\end{document}